\documentclass[graybox]{svmult}
\usepackage{mathptmx}       
\usepackage{helvet}         
\usepackage{courier}        
\usepackage{type1cm}        
\usepackage{makeidx}         
\usepackage{graphicx}        
\usepackage{multicol}        
\usepackage[bottom]{footmisc}

\usepackage{amsmath}
\usepackage{amssymb}
\usepackage{amscd}
\usepackage{indentfirst}
\numberwithin{equation}{section}
\newtheorem{thm}{Theorem}[section]
\newtheorem{prop}[thm]{Proposition}
\newtheorem{lem}[thm]{Lemma}

\DeclareMathOperator{\dom}{dom}

\def\1{{\bf 1}}

\usepackage{color}

\makeindex             

\begin{document}

\title*{Trotter-Kato product formulae in Dixmier ideal\\
\vspace{0.5cm}
\normalsize \it On the occasion of the 100th birthday of Tosio Kato}

\titlerunning{Trotter-Kato product formulae in Dixmier ideal}

\author{Valentin A.Zagrebnov}
\authorrunning{V.A.Zagrebnov }
\institute{V.A.Zagrebnov  \at Institut de Math\'{e}matiques de Marseille (UMR 7373) - AMU,
Centre de Math\'{e}matiques et Informatique - Technop\^{o}le Ch\^{a}teau-Gombert -
39, rue F. Joliot Curie, 13453 Marseille Cedex 13, France  \\
\email{Valentin.Zagrebnov@univ-amu.fr}}

\maketitle


\abstract{It is shown that for a certain class of the Kato functions the Trotter-Kato product formulae
converge in {Dixmier ideal} $\mathcal{C}_{1, {\infty}}$ in topology, which is defined by the
$\|\cdot\|_{1, \infty}$-norm. Moreover, the rate of convergence in this topology inherits the error-bound
estimate for the corresponding operator-norm convergence.}

\numberwithin{equation}{section}
\renewcommand{\theequation}{\arabic{section}.\arabic{equation}}
\section{Preliminaries. Symmetrically-normed ideals }\label{S0}
Let $\mathcal{H}$ be a separable Hilbert space.
For the first time the Trotter-Kato product formulae in {Dixmier ideal}
$\mathcal{C}_{1, {\infty}}(\mathcal{H})$, were shortly discussed in conclusion of the paper \cite{NeiZag1999}.
This remark was a program addressed to extension of results, which were known for the von Neumann-Schatten ideals
$\mathcal{C}_{p}(\mathcal{H})$, $p \geq 1$ since \cite{Zag1988}, \cite{NeiZag1990}.

Note that a subtle point of this program is the question about the rate of convergence in the corresponding
topology. Since the limit of the Trotter-Kato product formula is a strongly continuous semigroup, for the
von Neumann-Schatten ideals this topology is the trace-norm $\|\cdot\|_{1}$ on the trace-class ideal
$\mathcal{C}_{1}(\mathcal{H})$. In this case the limit is a Gibbs semigroup \cite{Zag2003}.

For self-adjoint Gibbs semigroups the rate of convergence was estimated for the first time
in \cite{DoIchTam1998} and \cite{IchTam1998}. The authors considered the case of the Gibbs-Schr\"{o}dinger
semigroups. They scrutinised in these papers a dependence of the rate of convergence for the (exponential)
Trotter formula on the smoothness of the potential in the Schr\"{o}dinger generator.

The first abstract result in this direction was due to \cite{NeiZag1999}. In this paper a general scheme of
\textit{lifting} the operator-norm rate convergence for the Trotter-Kato product formulae was proposed
and advocated for estimation the rate of the trace-norm convergence. This scheme was then improved and
extended in \cite{CacZag2001} to the case of \textit{nonself-adjoint} Gibbs semigroups.

The aim of the present note is to elucidate the question about the existence of other
then the von Neumann-Schatten proper two-sided ideals $\mathfrak{I}(\mathcal{H})$
of $\mathcal{L}(\mathcal{H})$ and then to prove the (non-exponential) Trotter-Kato product formula in
topology of these ideals together with estimate of the corresponding rate of convergence.
Here a particular case of the Dixmier ideal $\mathcal{C}_{1, {\infty}}(\mathcal{H})$ \cite{Dix1981},
\cite{Con1994}, is considered.
To specify this ideal we recall in Section \ref{S1} the notion of \textit{singular} trace and then of
the Dixmier trace \cite{Dix1966}, \cite{CarSuk2006}, in Section \ref{S2}. Main results about the Trotter-Kato
product formulae in the Dixmier ideal $\mathcal{C}_{1, {\infty}}(\mathcal{H})$ are collected in
Section \ref{S3}. There the arguments based on the lifting scheme \cite{NeiZag1999} (Theorem 5.1) are refined
for proving the Trotter-Kato product formulae convergence in the $\|\cdot\|_{1, \infty}$-topology with the
rate, which is \textit{inherited} from the operator-norm convergence.

To this end, in the rest of the present section we recall an important auxiliary tool: the concept of
\textit{symmetrically-normed} ideals, see e.g. \cite{GohKre1969}, \cite{Sim2005}.

Let $c_0 \subset l^{\infty}(\mathbb{N})$ be the subspace of bounded sequences
$\xi = \{\xi_j\}^\infty_{j=1} \in l^{\infty}(\mathbb{N})$ of real numbers, which tend to \textit{zero}.
We denote by $c_f$ the subspace of $c_0$ consisting of all sequences with \textit{finite} number of non-zero
terms (\textit{finite sequences}).
\begin{definition}\label{def6-2.1}
A real-valued function $\phi: \xi \mapsto \phi(\xi)$ defined on $c_f$ is called a {\em norming function}
\index{function!(symmetric) norming}
if it has the following properties:
\begin{eqnarray}
& & \phi(\xi)         >   0, \qquad \forall \xi \in c_f, \quad \xi \not= 0, \label{6-2.1}\\
& & \phi(\alpha\xi)      =   |\alpha|\phi(\xi), \qquad \forall \xi \in c_f, \quad \forall \alpha \in \mathbb{R},
\label{6-2.2}\\
& & \phi(\xi + \eta) \le  \phi(\xi) + \phi(\eta), \qquad \forall \xi,\eta \in c_f, \label{6-2.3}\\
& & \phi(1,0,\ldots)  =   1. \label{6-2.4}
\end{eqnarray}
A norming function $\phi$ is called to be {\em symmetric} if it has the
additional property
\begin{equation}\label{6-2.5}
\phi(\xi_1,\xi_2,...,\xi_n,0,0,\ldots) = \phi(|\xi_{j_1}|,|\xi_{j_2}|,...,|\xi_{j_n}|,0,0,\ldots)
\end{equation}
for any $\xi \in c_f$ and any permutation $j_1,j_2,\ldots , j_n$ of integers $1,2,\ldots , n$.
\end{definition}

It turns out that for any {\em symmetric norming} function $\phi$
and for any elements $\xi,\eta$ from the positive \textit{cone} $c^+$ of non-negative, non-increasing
sequences
such that $\xi,\eta \in c_f$ obey $\xi_1 \ge \xi_2 \ge \ldots \ge 0$,
$\eta_1 \ge \eta_2 \ge \ldots \ge 0$
and
\begin{equation}\label{6-2.5-1}
\sum^n_{j=1} \xi_j \le \sum^n_{j=1} \eta_j, \qquad n = 1,2,\ldots \ ,
\end{equation}
one gets the Ky Fan {inequality} \cite{GohKre1969} (Sec.3, \S3) :
\begin{equation}\label{6-2.5-2}
\phi(\xi) \le \phi(\eta) \ .
\end{equation}

Moreover, (\ref{6-2.5-2}) together with the properties (\ref{6-2.1}), (\ref{6-2.2}) and (\ref{6-2.4}) 
yield inequalities
\begin{equation}\label{6-2.5-3}
\xi_1 \le \phi(\xi) \le \sum^\infty_{j=1} \xi_j, \qquad \xi \in c_{f}^{+}:= c_{f}\cap c^{+}.
\end{equation}
Note that the left- and right-hand sides of (\ref{6-2.5-3}) are the simplest examples of symmetric
norming functions on domain $c_{f}^{+}$:
\begin{equation}\label{6-2.5-4}
\phi_{\infty}(\xi):= \xi_1 \ \ \ {\rm{and}} \ \ \ \phi_{1}(\xi):= \sum^\infty_{j=1} \xi_j \ .
\end{equation}
By Definition \ref{def6-2.1} the observations (\ref{6-2.5-3}) and (\ref{6-2.5-4}) yield
\begin{eqnarray}\label{6-2.5-5}
&&\phi_{\infty}(\xi):= \max_{j \geq 1}|\xi_j| \ , \ \  \phi_{1}(\xi):= \sum^\infty_{j=1} |\xi_j|
\ \ , \ \\
&&  \phi_{\infty}(\xi)\leq \phi(\xi) \leq \phi_{1}(\xi) \ , \ \ {\rm{for \ all}} \ \ \ \xi \in c_{f} \ .
\nonumber
\end{eqnarray}

We denote by $\xi^* := \{\xi^*_1,\xi^*_2,\ldots \ \}$ a decreasing \textit{rearrangement}:
$\xi^*_1 = \sup_{j\geq 1}|\xi_j|\ $, $\xi^*_1 + \xi^*_2 = \sup_{i\not=j}\{|\xi_i| + |\xi_j|\}, \ldots \ $,
of the sequence of absolute values $\{|\xi_n|\}_{n\geq 1}$, i.e., $\xi^*_1 \ge \xi^*_2 \ge \ldots \ $.
Then $\xi \in c_f$ implies $\xi^* \in c_f$ and by (\ref{6-2.5}) one obtains also that
\begin{equation}\label{6-2.10}
\phi(\xi) = \phi(\xi^*), \qquad \xi \in c_f \ .
\end{equation}
Therefore, any symmetric norming function $\phi$ is uniquely defined by its values on
the positive cone $c^+$.

Now, let $\xi = \{\xi_1,\xi_2,\ldots\} \in c_0$. We define
\begin{equation*}\label{6-2.6}
\xi^{(n)} := \{\xi_1,\xi_2,\ldots,\xi_n,0,0,\ldots \ \} \in c_{f} \ .
\end{equation*}
Then if $\phi$ is a symmetric norming function, we define
\begin{equation}\label{6-2.7}
c_\phi := \{\xi \in c_0: \sup_n\phi(\xi^{(n)}) < + \infty\}.
\end{equation}
Therefore, one gets
\begin{equation*}\label{6-2.8}
c_f \subseteq c_\phi \subseteq c_0 \subset l^{\infty}.
\end{equation*}

Note that by (\ref{6-2.5})-(\ref{6-2.5-2}) and  (\ref{6-2.7}) one gets
\begin{equation*}\label{6-2.8-1}
\phi(\xi^{(n)})\leq \phi(\xi^{(n+1)})\leq \sup_n\phi(\xi^{(n)})\ , \ {\rm{for \ any}} \ \xi \in c_\phi \ .
\end{equation*}
Then the limit
\begin{equation}\label{6-2.9}
\phi(\xi) := \lim_{n\to\infty} \phi(\xi^{(n)})\ , \qquad \xi \in c_\phi,
\end{equation}
exists and $\phi(\xi) = \sup_n\phi(\xi^{(n)})$, i.e. the symmetric norming function $\phi$ is a \textit{normal}
functional on the set $c_\phi$ (\ref{6-2.7}), which is a linear space over $\mathbb{R}$.

By virtue of (\ref{6-2.3}) and (\ref{6-2.5-5}) one also gets that any symmetric norming function
is \textit{continuous} on $c_f$:
\begin{equation*}\label{6-2.9-1}
|\phi(\xi) - \phi(\eta)| \leq \phi(\xi - \eta) \leq \phi_{1}(\xi - \eta) \ , \ \forall \xi,\eta \in c_f \ .
\end{equation*}

Suppose that $X$ is a compact operator, i.e. $X \in \mathcal{C}_{\infty}(\mathcal{H})$. Then we denote by
\begin{equation*}\label{6-2.11}
s(X) := \{s_1(X),s_2(X),\ldots \ \} ,
\end{equation*}
the sequence of \textit{singular values} of $X$ counting multiplicities.
We always assume that
\begin{equation*}\label{6-2.12}
s_1(X) \ge s_2(X) \ge \ldots \ge s_n(X) \ge \ldots \ .
\end{equation*}

To define \textit{symmetrically-normed} ideals of the compact operators $\mathcal{C}_{\infty}(\mathcal{H})$ we
introduce the notion of a symmetric norm.
\begin{definition}\label{def6-2.1-2}
Let $\mathfrak{I}$ be a two-sided ideal of $\mathcal{C}_{\infty}(\mathcal{H})$. A functional
$\|\cdot\|_{sym}: \mathfrak{I} \rightarrow \mathbb{R}^{+}_{0}$ is called a \textit{symmetric norm} if besides
the usual properties of the operator norm $\|\cdot\|$:
\begin{eqnarray*}
& & \|X\|_{sym}>0, \qquad \forall X \in \mathfrak{I} , \quad X \not= 0, \label{6-2.12-1}\\
& &  \|\alpha X\|_{sym}      =   |\alpha|\|X\|_{sym} \ , \qquad \forall X \in \mathfrak{I} ,
\quad \forall \alpha \in \mathbb{C}, \label{6-2.12-2}\\
& & \|X + Y\|_{sym}  \le  \|X\|_{sym}  + \|Y\|_{sym}\ , \ \forall X, Y \in \mathfrak{I} , \label{6-2.12-3}
\end{eqnarray*}
it verifies the following additional properties:
\begin{eqnarray}
& & \|A X B\|_{sym}        \leq   \|A\| \| X\|_{sym} \| B\|, \ X \in \mathfrak{I} , \
A,B \in \mathcal{L}(\mathcal{H}), \label{6-2.12-4}\\
& &  \|\alpha X\|_{sym}  =  |\alpha|\|X\| = |\alpha| \ s_1(X) , \
{\rm{for \ any \ rank-one \ operator}} \  X \in \mathfrak{I}. \label{6-2.12-5}
\end{eqnarray}

If the condition (\ref{6-2.12-4}) is replaced by
\begin{eqnarray}\label{6-2.12-4U}
&& \|U X\|_{sym}  = \|X U\|_{sym} = \| X\|_{sym} \ , \  X \in \mathfrak{I} \ , \\
&& \hspace{1cm} {\rm{for \ any \ unitary \ operator}} \ U \ {\rm{on}} \ \mathcal{H} \ , \nonumber
\end{eqnarray}
then, instead of the symmetric norm, one gets definition of \textit{invariant norm} $\|\cdot\|_{inv}$.
\end{definition}

First, we note that the ordinary operator norm $\|\cdot\|$ on any ideal
$\mathfrak{I}\subseteq \mathcal{C}_{\infty}(\mathcal{H})$ is evidently a symmetric norm as well as
an invariant norm.

Second, we observe that in fact, every symmetric norm is invariant. Indeed, for any unitary operators $U$ and
$V$ one gets by (\ref{6-2.12-4}) that
\begin{equation}\label{6-2.12-51}
\|U X V\|_{sym} \leq  \| X\|_{sym} \ , \  X \in \mathfrak{I} \ .
\end{equation}
Since $X = U^{-1}U X V V^{-1}$, we also get $\| X\|_{sym} \leq \|U X V\|_{sym}$, which together with
(\ref{6-2.12-51}) yield (\ref{6-2.12-4U}).

Third, we claim that $ \| X\|_{sym} =  \| X^\ast\|_{sym}$. Let $X = U |X|$ be the polar
representation of the operator $X \in \mathfrak{I}$. Since $U^\ast X = |X|$, then by (\ref{6-2.12-4U}) 
we obtain
$\| X\|_{sym}=\||X|\|_{sym}$. The same line of reasoning applied to the adjoint operator $X^\ast = |X| U^\ast$
yields $\| X^\ast\|_{sym}=\||X|\|_{sym}$, that proves the claim.

Now we can apply the concept of the symmetric norming functions to describe the symmetrically-normed ideals
of the unital algebra of bounded operators $\mathcal{L}(\mathcal{H})$, or in general, the symmetrically-normed
ideals generated by symmetric norming functions. Recall that any \textit{proper} two-sided ideal
$\mathfrak{I}(\mathcal{H})$ of $\mathcal{L}(\mathcal{H})$ is contained in compact operators
$\mathcal{C}_{\infty}(\mathcal{H})$ and contains the set $\mathcal{K}(\mathcal{H})$ of finite-rank operators,
see e.g. \cite{Piet2014}, \cite{Sim2005}:
\begin{equation}\label{6-2.12-5-1}
\mathcal{K}(\mathcal{H}) \subseteq \mathfrak{I}(\mathcal{H}) \subseteq \mathcal{C}_{\infty}(\mathcal{H}) \ .
\end{equation}

To clarify the relation between symmetric norming functions and the symmetrically-normed ideals
we mention that there is an obvious one-to-one correspondence between functions $\phi$ (Definition
\ref{def6-2.1}) on the cone $c^{+}$ and the symmetric norms $\|\cdot\|_{sym}$ on
$\mathcal{K}(\mathcal{H})$. To proceed with a general setting one needs definition of the following
relation.
\begin{definition}\label{def6-2.3}
Let $c_\phi$ be the set of vectors (\ref{6-2.7}) generated by a symmetric norming function $\phi$. We
associate with $c_\phi$ a subset of compact operators
\begin{equation}\label{6-2.13}
\mathcal{C}_{\phi}(\mathcal{H}) := \{X \in \mathcal{C}_{\infty}(\mathcal{H}): s(X) \in c_\phi\} \ .
\end{equation}
\end{definition}

This definition implies that the set $\mathcal{C}_{\phi}(\mathcal{H})$ is a proper two-sided ideal of 
the algebra $\mathcal{L}(\mathcal{H})$ of all bounded operators on $\mathcal{H}$. Setting, see (\ref{6-2.9}),
\begin{equation}\label{6-2.14}
\|X\|_\phi := \phi(s(X))\ , \qquad X \in \mathcal{C}_{\phi}(\mathcal{H}) \ ,
\end{equation}
one obtains the \textit{symmetric norm}: $\|\cdot\|_{sym} = \|\cdot\|_\phi \ $, on the ideal
$\mathfrak{I} = \mathcal{C}_{\phi}(\mathcal{H})$ (Definition \ref{def6-2.1-2}) such that
this symmetrically-normed ideal becomes a Banach space. Then in accordance with (\ref{6-2.12-5-1}) and
(\ref{6-2.13}) we obtain by (\ref{6-2.5-5}) that
\begin{equation}\label{6-2.14-1}
\mathcal{K}(\mathcal{H}) \subseteq \mathcal{C}_{1}(\mathcal{H}) \subseteq \mathcal{C}_{\phi}(\mathcal{H})
\subseteq \mathcal{C}_{\infty}(\mathcal{H})\ .
\end{equation}
Here the trace-class operators $\mathcal{C}_{1}(\mathcal{H}) :=  \mathcal{C}_{\phi_1}(\mathcal{H})$,
where the symmetric norming function $\phi_1$ is defined in (\ref{6-2.5-4}), and
\begin{equation*}\label{6-2.14-2}
\|X\|_\phi \le \|X\|_1 \ , \qquad X \in \mathcal{C}_{1}(\mathcal{H})\ .
\end{equation*}
%
%
\begin{remark}\label{rem6-2.1-0}
By virtue of inequality (\ref{6-2.5-2}) and by definition of symmetric norm (\ref{6-2.14})
the so-called {\em dominance} property
holds: if $X \in \mathcal{C}_{\phi}(\mathcal{H})$, $Y \in \mathcal{C}_{\infty}(\mathcal{H})$ and
\begin{equation*}\label{6-2.14-3}
\sum^n_{j=1}s_j(Y) \le \sum^n_{j=1}s_j(X)\ , \qquad n =1,2,\ldots \ ,
\end{equation*}
then $Y \in \mathcal{C}_{\phi}(\mathcal{H})$ and $\|Y\|_\phi \le \|X\|_\phi$.
\end{remark}
\begin{remark}\label{rem6-2.1}
To distinguish in (\ref{6-2.14-1}) \textit{nontrivial} ideals $\mathcal{C}_{\phi}$ one needs some
criteria based on the properties of $\phi$ or of the norm $\|\cdot\|_\phi$.
For example, any symmetric norming function (\ref{6-2.10}) defined by
\begin{equation*}\label{6-2.14-4}
\phi^{(r)}(\xi) := \sum_{j=1}^{r} \ \xi^*_j \ , \qquad \xi \in c_f \ ,
\end{equation*}
generates for arbitrary fixed $r\in \mathbb{N}$ the symmetrically-normed ideals, which are trivial
in the sense that $\mathcal{C}_{\phi^{(r)}}(\mathcal{H}) = \mathcal{C}_{\infty}(\mathcal{H})$.
Criterion for an operator $A$ to belong to a nontrivial ideal $\mathcal{C}_{\phi}$ is
\begin{equation}\label{6-2.14-5}
M = \sup_{m\geq 1} \|P_m A P_m\|_\phi < \infty \ ,
\end{equation}
where $\{P_m\}_{m\geq 1}$ is a monotonously increasing sequence of the finite-dimensional orthogonal
projectors on $\mathcal{H}$ strongly convergent to the identity operator \cite{GohKre1969}. Note that for
$A \in \mathcal{C}_{\infty}$ the condition (\ref{6-2.14-5}) is trivial.
\end{remark}

We consider now a couple of examples to elucidate the concept of the symmetrically-normed
ideals $\mathcal{C}_{\phi}(\mathcal{H})$ generated by the symmetric norming functions $\phi$ and the r\^{o}le
of the functional \textit{trace} on these ideals.
\begin{example}\label{ex6-2.1}
The von Neumann-Schatten ideals $\mathcal{C}_{p}(\mathcal{H})$ \cite{Schat1970}. These ideals of
$\mathcal{C}_{\infty}(\mathcal{H})$ are generated by symmetric norming functions $\phi(\xi) := \|\xi\|_{p}$,
where
\begin{equation*}\label{6-2.15}
\|\xi\|_{p} = \left(\sum^\infty_{j=1} |\xi_j|^p\right)^{1/p}, \qquad \xi \in c_f,
\end{equation*}
for $1 \le p < +\infty$,  and by
\begin{equation*}\label{6-2.16}
\|\xi\|_{\infty} = \sup_j|\xi_j|, \qquad \xi \in c_f,
\end{equation*}
for $p = +\infty$.
Indeed, if we put $\{\xi_{j}^\ast := s_j(X)\}_{j\geq 1}$, for $X\in \mathcal{C}_{\infty}(\mathcal{H})$,
then the symmetric norm $\|X\|_{\phi} = \|s(X)\|_{p}$ coincides with
$\|X\|_{p}$ and the corresponding symmetrically-normed ideal $\mathcal{C}_{\phi}(\mathcal{H})$
is identical to the von~Neumann-Schatten class $\mathcal{C}_{p}(\mathcal{H})$.

By definition, for any $X \in \mathcal{C}_{p}(\mathcal{H})$ the
trace: $|X|\mapsto {\rm{Tr}} |X| = \sum_{j\geq 1} s_j(X) \geq 0$. The trace norm
$\|X\|_{1} = {\rm{Tr}} |X|$ is finite on the trace-class operators $\mathcal{C}_{1}(\mathcal{H})$ and
it is \textit{infinite} for $X \in \mathcal{C}_{p>1}(\mathcal{H})$. We say that for $p>1$ the
von Neumann-Schatten ideals admit \textit{no} trace, whereas for $p=1$ the map: $X \mapsto {\rm{Tr}}\, X$,
exists and it is continuous in the $\|\cdot\|_{1}$-topology.

Note that by virtute of the $\rm{Tr}$-linearity the trace norm:
$\mathcal{C}_{1, +}(\mathcal{H})\ni X \mapsto \|X\|_{1}$ is \textit{linear}
on the positive cone $\mathcal{C}_{1, +}(\mathcal{H})$ of the trace-class operators.
\end{example}
\begin{example}\label{ex6-2.2}
Now we consider symmetrically-normed ideals $\mathcal{C}_{\Pi}(\mathcal{H})$.
To this aim let $\Pi = \{\pi_j\}^\infty_{j=1} \in c^+$ be a non-increasing sequence of
positive numbers with $\pi_1 = 1$. We associate with $\Pi$ the function
\begin{equation}\label{6-2.17}
\phi_\Pi(\xi) = \sup_n\left\{\frac{1}{\sum^n_{j=1}\pi_j}\sum^n_{j=1}\xi^*_j\right\}, \qquad \xi \in c_f.
\end{equation}
It turns out that $\phi_\Pi$ is a symmetric norming function. Then the corresponding to (\ref{6-2.7}) set
$c_{\phi_\Pi}$ is defined by
\begin{equation*}\label{6-2.18}
c_{\phi_\Pi} := \left\{\xi \in c_f: \sup_n\frac{1}{\sum^n_{j=1}\pi_j}\sum^n_{j=1}\xi^*_j < +\infty \right\} \ .
\end{equation*}
Hence, the two-sided symmetrically-normed ideal
$\mathcal{C}_{\Pi}(\mathcal{H}):= \mathcal{C}_{\phi_\Pi}(\mathcal{H})$
generated by symmetric norming function (\ref{6-2.17}) consists of all those compact operators $X$ that
\begin{equation}\label{6-2.19}
\|X\|_{\phi_\Pi} := \sup_n\frac{1}{\sum^n_{j=1}\pi_j}\sum^n_{j=1} s_j(X) < +\infty \ .
\end{equation}
This equation defines a symmetric norm $\|X\|_{sym}=\|X\|_{\phi_\Pi}$ on the ideal
$\mathcal{C}_{\Pi}(\mathcal{H})$, see Definition \ref{def6-2.1-2}.
%
%

Now let $\Pi = \{\pi_j\}^\infty_{j=1}$, with $\pi_1 =1$, satisfy
\begin{equation}\label{6-2.20}
\sum^\infty_{j=1} \pi_j = +\infty \qquad \mbox{and} \qquad \lim_{j\to\infty} \pi_j = 0 \ .
\end{equation}
Then the ideal $\mathcal{C}_{\Pi}(\mathcal{H})$ is \textit{nontrivial}:
$\mathcal{C}_{\Pi}(\mathcal{H}) \not= \mathcal{C}_{\infty}(\mathcal{H})$ and
$\mathcal{C}_{\Pi}(\mathcal{H}) \not= \mathcal{C}_{1}(\mathcal{H})$, see Remark \ref{rem6-2.1}, and one has
\begin{equation}\label{6-2.20-1}
\mathcal{C}_{1}(\mathcal{H}) \subset \mathcal{C}_{\Pi}(\mathcal{H}) \subset
\mathcal{C}_{\infty}(\mathcal{H}) \ .
\end{equation}

If in addition to (\ref{6-2.20}) the sequence $\Pi = \{\pi_j\}^\infty_{j=1}$ is \textit{regular},
i.e. it obeys
\begin{equation}\label{6-2.21}
\sum^n_{j=1} \pi_j = O(n\pi_n) \ , \qquad n \rightarrow \infty \ ,
\end{equation}
then $X \in \mathcal{C}_{\Pi}(\mathcal{H})$  {if and only if}
\begin{equation}\label{6-2.22}
s_n(X) = O(\pi_n) \ , \qquad n \rightarrow \infty \ ,
\end{equation}
cf. condition (\ref{6-2.14-5}). On the other hand, the asymptotics
\begin{equation*}\label{6-2.22-1}
s_n(X) = o(\pi_n) \ , \qquad n \rightarrow \infty \ ,
\end{equation*}
{{implies that $X$ belongs to:
\begin{equation*}
\mathcal{C}_{\Pi}^{0}(\mathcal{H}):= \{X \in \mathcal{C}_{\Pi}(\mathcal{H}):
\lim_{n \rightarrow \infty}\frac{1}{\sum^n_{j=1}\pi_j}\sum^n_{j=1} s_j(X) = 0 \},
\end{equation*}
such that $\mathcal{C}_{1}(\mathcal{H})\subset \mathcal{C}_{\Pi}^{0}(\mathcal{H}) \subset 
\mathcal{C}_{\Pi}(\mathcal{H})$.}}
\end{example}
\begin{remark}\label{rem6-2.2}
A natural choice of the sequence $\{\pi_j\}^\infty_{j=1}$ that satisfies (\ref{6-2.20}) is
$\pi_j = j^{-\alpha}$, $0 < \alpha \le 1$. Note that if $0 < \alpha < 1$, then the sequence
$\Pi = \{\pi_j\}^\infty_{j=1}$ satisfies (\ref{6-2.21}), i.e. it is regular for
$\varepsilon = 1 -\alpha$.
Therefore, the two-sided symmetrically-normed ideal $\mathcal{C}_{\Pi}(\mathcal{H})$ generated by symmetric
norming function (\ref{6-2.17}) consists of all those compact operators $X$, which singular values obey
(\ref{6-2.22}):
\begin{equation}\label{6-2.23}
s_n(X) = O(n^{-\alpha}), \quad 0 < \alpha < 1, \quad n \rightarrow \infty \ .
\end{equation}
Let $\alpha = 1/p \, , \ p>1$. Then the corresponding to (\ref{6-2.23}) symmetrically-normed ideal defined by
\begin{equation*}\label{6-2.23-1}
\mathcal{C}_{p,\infty}(\mathcal{H}) := \{X \in \mathcal{C}_{\infty}(\mathcal{H}):
s_n(X) = O(n^{-1/p}), \  p>1 \} \ ,
\end{equation*}
is known as the \textit{weak}-$\mathcal{C}_{p}$ ideal \cite{Piet2014}, \cite{Sim2005}.
\end{remark}

Whilst by virtue of (\ref{6-2.23}) the weak-$\mathcal{C}_{p}$ ideal admit no trace, definition (\ref{6-2.19})
implies that for the regular case $p > 1$ a symmetric norm on $\mathcal{C}_{p,\infty}(\mathcal{H})$ is
equivalent to
\begin{equation}\label{6-2.24}
\|X\|_{p,\infty} = \sup_n \frac{1}{n^{1- 1/p}}\sum^n_{j=1}s_j(X) \ ,
\end{equation}
and it is obvious that $\mathcal{C}_{1}(\mathcal{H}) \subset \mathcal{C}_{p,\infty}(\mathcal{H}) \subset
\mathcal{C}_{\infty}(\mathcal{H})$. {{Taking into account the H\"{o}lder inequality one can to refine these 
inclusions for $1\leq q \leq p$ as follows:
$\mathcal{C}_{1}(\mathcal{H})\subseteq \mathcal{C}_{q}(\mathcal{H}) \subseteq 
\mathcal{C}_{p,\infty}(\mathcal{H}) \subset \mathcal{C}_{\infty}(\mathcal{H})$.}} 
\section{Singular traces }\label{S1}

{{Note that (\ref{6-2.24}) implies:
$\mathcal{C}_{1}(\mathcal{H})\ni A \mapsto \|A\|_{p,\infty} < \infty$, but any related 
to the ideal $\mathcal{C}_{p,\infty}(\mathcal{H})$ linear, positive,
and unitarily invariant functional (\textit{trace}) is {zero} on the 
set of finite-rank operators $\mathcal{K}(\mathcal{H})$, or trivial.}}
We remind that these not \textit{normal} traces:
\begin{equation}\label{6-2.24-1}
{\rm{Tr}} {_{\omega}}(X):= \omega (\{n^{-1 + 1/p}\sum^{n}_{j=1}\ s_j(X)\}_{n=1}^{\infty}) \ ,
\end{equation}
are called \textit{singular}, \cite{Dix1966}, \cite{LoSuZa2013}. Here $\omega$ is an \textit{appropriate} 
linear positive normalised functional (\textit{state}) on the Banach space $l^{\infty}(\mathbb{N})$ of 
bounded sequences. Recall that the set of the states 
$\mathcal{S}(l^{\infty}(\mathbb{N}))\subset (l^{\infty}(\mathbb{N}))^*$,
where $(l^{\infty}(\mathbb{N}))^*$ is dual of the Banach space $l^{\infty}(\mathbb{N})$.
The singular trace (\ref{6-2.24-1}) is continuous in topology defined by the norm (\ref{6-2.24}).
\begin{remark}\label{rem6-2.3}
(a) The \textit{weak}-$\mathcal{C}_{p}$  ideal, which is defined for $p=1$ by
\begin{equation}\label{6-2.25}
\mathcal{C}_{1,\infty}(\mathcal{H}) := \{X \in \mathcal{C}_{\infty}(\mathcal{H}): \sum^n_{j=1}s_j(X) =
O(\ln(n)),
\ n \rightarrow \infty\} \ ,
\end{equation}
has a special interest. Note that since $\Pi = \{j^{-1}\}^\infty_{j=1}$ does {not} satisfy
(\ref{6-2.21}), the characterisation $s_n(X) = O(n^{-1})$, is \textit{not} true,
see (\ref{6-2.22}), (\ref{6-2.23}). In this case the equivalent norm can be defined on the
ideal (\ref{6-2.25}) as
\begin{equation}\label{6-2.26}
\|X\|_{1,\infty} := \sup_{n \in \mathbb{N}}\frac{1}{1 +\ln(n)}\sum^n_{j=1}s_j(X) \ .
\end{equation}
By, virtute of (\ref{6-2.20-1}) and Remark \ref{rem6-2.2} one gets that
$\mathcal{C}_{1}(\mathcal{H}) \subset \mathcal{C}_{1,\infty}(\mathcal{H})$ and that
$\mathcal{C}_{1}(\mathcal{H})\ni A \mapsto \|A\|_{1,\infty} < \infty$.

(b) In contrast to linearity of the trace-norm $\|\cdot\|_{1}$ on the positive cone 
$\mathcal{C}_{1, +}(\mathcal{H})$, see Example \ref{ex6-2.1}, the map $X \mapsto \|X\|_{1,\infty}$ on the 
positive cone $\mathcal{C}_{1,\infty, +}(\mathcal{H})$ is \textit{not} linear. Although this map is 
homogeneous: $\alpha A \mapsto \alpha \|A\|_{1,\infty}$, $\alpha \geq 0$,   
for $A,B \in \mathcal{C}_{1,\infty, +}(\mathcal{H})$ one gets that in general 
$\|A + B\|_{1,\infty} \neq \|A\|_{1,\infty} + \|B\|_{1,\infty}$. 

But it is known that on the space $l^{\infty}(\mathbb{N})$
there exists a state $\omega \in \mathcal{S}(l^{\infty}(\mathbb{N}))$  such that the map
\begin{equation}\label{6-2.26-1}
X \mapsto {\rm{Tr}} {_{\omega}}(X):= \omega (\{(1 +\ln(n))^{-1}\sum^{n}_{j=1}\ s_j(X)\}_{n=1}^{\infty} ) \ ,
\end{equation}
is \textit{linear} and verifies the properties of the (singular) \textit{trace} for any
$X\in \mathcal{C}_{1,\infty}(\mathcal{H})$. We construct $\omega$ in Section \ref{S2}.
This particular choice of the state $\omega$ defines the {\em Dixmier trace} on the space
$\mathcal{C}_{1,\infty}(\mathcal{H})$, which is called, in turn, the \textit{Dixmier ideal}, see e.g.
\cite{CarSuk2006}, \cite{Con1994}. The Dixmier trace (\ref{6-2.26-1}) is obviously continuous in topology
defined by the norm (\ref{6-2.26}).
This last property is basic for discussion in Section \ref{S3} of the Trotter-Kato product formula in the
$\|\cdot\|_{p,\infty}$-topology, for $p \geq 1$.
\end{remark}
\begin{example}\label{ex6-2.3}
With non-increasing sequence of positive numbers $\pi = \{\pi_j\}^\infty_{j=1}$, $\pi_1 =1$, one can
associate the symmetric norming function $\phi_\pi$ given by
\begin{equation*}\label{6-2.27}
\phi_\pi(\xi) := \sum^\infty_{j=1}\pi_j\xi^*_j \ , \qquad \xi \in c_f \ .
\end{equation*}
The corresponding symmetrically-normed ideal we denote by
$\mathcal{C}_{\pi}(\mathcal{H}):= \mathcal{C}_{\phi_\pi}(\mathcal{H})$.

If the sequence $\pi$ satisfies (\ref{6-2.20}), then ideal $\mathcal{C}_{\pi}(\mathcal{H})$ does not
coincide neither with $\mathcal{C}_{\infty}(\mathcal{H})$ nor with $\mathcal{C}_{1}(\mathcal{H})$. If,
in particular, $\pi_j = j^{-\alpha}$, $j = 1,2,\ldots \ $, for $0 < \alpha \le 1$, then the corresponding
ideal is denoted by $\mathcal{C}_{\infty,p}(\mathcal{H})$, $p = 1/\alpha$. The norm on this ideal is
given by
\begin{equation*}\label{6-2.28}
\|X\|_{\infty,p} := \sum^\infty_{j=1} j^{-1/p}\ s_j(X) \ , \ \  \  \ p\in [1, \infty) \ .
\end{equation*}
The symmetrically-normed ideal $\mathcal{C}_{\infty,1}(\mathcal{H})$ is called the {\em Macaev ideal}
\cite{GohKre1969}.
It turns out that the Dixmier ideal $\mathcal{C}_{1, \infty}(\mathcal{H})$ is dual of the Macaev ideal:
$\mathcal{C}_{1, \infty}(\mathcal{H}) = \mathcal{C}_{\infty,1}(\mathcal{H})^* $.
\end{example}
%
\begin{prop}\label{pro7-1.1}
The space $\mathcal{C}_{1,\infty}(\mathcal{H})$ endowed by the norm $\|\cdot\|_{1,\infty}$ is
a Banach space.
\end{prop}
\noindent The proof is quite standard although tedious and long. We address the readers to the corresponding
references, e.g. \cite{GohKre1969}.
\begin{prop}\label{pro7-1.2} The space $\mathcal{C}_{1,\infty}(\mathcal{H})$ endowed by the norm
$\|\cdot\|_{1,\infty}$  is a Banach ideal in the algebra of bounded operators $\mathcal{L}(\mathcal{H})$.
\end{prop}
\begin{proof} To this end it is sufficient to prove that if $A$ and $C$ are bounded operators, then
$B \in\mathcal{C}_{1,\infty}(\mathcal{H})$ implies $A B C \in \mathcal{C}_{1,\infty}(\mathcal{H})$.
Recall that singular values of the operator $A B C$ verify the estimate $s_j(A B C)\leq \|A\| \|C\| s_j(B)$.
By (\ref{6-2.26}) it yields
\begin{eqnarray}\label{7-1.1-0}
&& \|ABC\|_{1,\infty} = \sup_{n \in \mathbb{N}}\frac{1}{1 +\ln(n)}\sum^n_{j=1}s_j(A B C) \leq \\
&& \|A\| \|C\| \sup_{n \in \mathbb{N}}\frac{1}{1 +\ln(n)}\sum^n_{j=1}s_j(B) = \|A\| \|C\|  \|B\|_{1,\infty} \ ,
\nonumber
\end{eqnarray}
and consequently the proof of the assertion.
\hfill $\square$  \end{proof}

Recall that for any $A \in \mathcal{L}(\mathcal{H})$ and all $B \in\mathcal{C}_{1}(\mathcal{H})$ one can
define a linear functional on $\mathcal{C}_{1}(\mathcal{H})$ given by ${{\rm{Tr}}}_{\mathcal{H}} (A B)$.
The set of these functionals
$\{{{\rm{Tr}}}_{\mathcal{H}} (A  \cdot)\}_{A \in \mathcal{L}(\mathcal{H})}$
is just the \textit{dual} space $\mathcal{C}_{1}(\mathcal{H})^*$ of $\mathcal{C}_{1}(\mathcal{H})$ with the
operator-norm topology. In other words, $\mathcal{L}(\mathcal{H})=\mathcal{C}_{1}(\mathcal{H})^* $, in the
sense that the map $A \mapsto {{\rm{Tr}}}_{\mathcal{H}} (A  \cdot)$ is the isometric isomorphism of
$\mathcal{L}(\mathcal{H})$ onto $\mathcal{C}_{1}(\mathcal{H})^*$.

With help of the \textit{duality} relation
\begin{equation}\label{7-1.1}
\langle A| B \rangle : ={{\rm{Tr}}}_{\mathcal{H}} (A B) \ ,
\end{equation}
one can also describe the space $\mathcal{C}_{1}(\mathcal{H})_{*}$, which is a \textit{predual} of
$\mathcal{C}_{1}(\mathcal{H})$, i.e., its dual
$(\mathcal{C}_{1}(\mathcal{H})_{*})^* =\mathcal{C}_{1}(\mathcal{H})$.
To this aim for each \textit{fixed} $B \in\mathcal{C}_{1}(\mathcal{H})$ we consider
the functionals $A \mapsto {{\rm{Tr}}}_{\mathcal{H}} (A B)$ on $\mathcal{L}(\mathcal{H})$.
It is known that they are not \textit{all} continuous linear functional on bounded operators
$\mathcal{L}(\mathcal{H})$, i.e., $\mathcal{C}_{1}(\mathcal{H}) \subset \mathcal{L}(\mathcal{H})^*$,
but they yield the entire dual only of compact operators,
i.e., $\mathcal{C}_{1}(\mathcal{H})=\mathcal{C}_{\infty}(\mathcal{H})^* $. Hence,
$\mathcal{C}_{1}(\mathcal{H})_{*} = \mathcal{C}_{\infty}(\mathcal{H})$.

Now we note that under duality relation (\ref{7-1.1}) the Dixmier ideal
$\mathcal{C}_{1,\infty}(\mathcal{H})$ is the dual of the Macaev ideal:
$\mathcal{C}_{1,\infty}(\mathcal{H}) = \mathcal{C}_{\infty, 1}(\mathcal{H})^*$, where
\begin{equation}\label{7-1.2}
\mathcal{C}_{\infty, 1} (\mathcal{H})= \{A \in \mathcal{C}_{\infty}(\mathcal{H}):
\sum_{n\geq1} \frac{1}{n} \ s_{n}(A) < \infty \} \ ,
\end{equation}
see Example \ref{ex6-2.3}.
By the same duality relation and by similar calculations one also obtains that the \textit{predual} of
$\mathcal{C}_{\infty, 1}(\mathcal{H})$  is the ideal
$\mathcal{C}_{\infty, 1}(\mathcal{H})_* = \mathcal{C}_{1, \infty}^{(0)}(\mathcal{H})$, defined by
\begin{equation}\label{7-1.3}
\mathcal{C}_{1, \infty}^{(0)} (\mathcal{H}): = \{A \in \mathcal{C}_{\infty}(\mathcal{H}):
\sum_{j \geq1}^{n} \ s_{j}(A) = o(\ln (n)), \ n \rightarrow \infty\} \ .
\end{equation}
By virtue of  (\ref{6-2.25}) (see Remark \ref{rem6-2.3}) the ideal (\ref{7-1.3}) is not self-dual since
\begin{equation*}
\mathcal{C}_{1, \infty}^{(0)}(\mathcal{H})^{**} = \mathcal{C}_{1,\infty}(\mathcal{H})\supset
\mathcal{C}_{1, \infty}^{(0)}(\mathcal{H}).
\end{equation*}

The problem which has motivated construction of the Dixmier trace in \cite{Dix1966} was related to
the question of a general definition of the \textit{trace}, i.e. a linear, positive, and unitarily invariant
functional on a \textit{proper} Banach ideal $\mathfrak{I}(\mathcal{H})$ of the unital algebra of bounded
operators $\mathcal{L}(\mathcal{H})$.
Since any {proper} two-sided ideal $\mathfrak{I}(\mathcal{H})$ of $\mathcal{L}(\mathcal{H})$ is contained
in compact operators $\mathcal{C}_{\infty}(\mathcal{H})$ and contains the set $\mathcal{K}(\mathcal{H})$ of
finite-rank operators ((\ref{6-2.12-5-1}), Section \ref{S0}), \textit{domain} of definition of the {trace} has
to coincide with the ideal $\mathfrak{I}(\mathcal{H})$.

\begin{remark}\label{rem6-2.4}
The \textit{canonical} trace ${\rm{Tr}}_{\mathcal{H}}(\cdot)$ is nontrivial only on domain, which is the
trace-class ideal $\mathcal{C}_{1}(\mathcal{H})$, see Example \ref{ex6-2.1}. We recall that it is
characterised by the property of \textit{normality}:
${\rm{Tr}}_{\mathcal{H}}(\sup_{\alpha} B_{\alpha}) = \sup_{\alpha}{\rm{Tr}}_{\mathcal{H}}(B_{\alpha})$, for
every directed increasing bounded family $\{B_{\alpha}\}_{\alpha \in \Delta}$ of positive operators from
$\mathcal{C}_{1, +}(\mathcal{H})$.

Note that every nontrivial \textit{normal} trace on $\mathcal{L}(\mathcal{H})$
is proportional to the canonical trace ${\rm{Tr}}_{\mathcal{H}}(\cdot)$, see e.g. \cite{Dix1981},
\cite{Piet2014}. Therefore, the Dixmier trace (\ref{6-2.26-1}) :
$\mathcal{C}_{1, \infty} \ni X \mapsto {\rm{Tr}} {_{\omega}}(X)$, is \textit{not} normal.
\end{remark}
\begin{definition}\label{def7-1.1}
A \textit{trace} on the proper Banach ideal $\mathfrak{I}(\mathcal{H})\subset \mathcal{L}(\mathcal{H})$ is called
\textit{singular} if it vanishes on the set $\mathcal{K}(\mathcal{H})$.
\end{definition}

Since a singular trace is defined up to trace-class operators $\mathcal{C}_{1}(\mathcal{H})$,
then by Remark \ref{rem6-2.4} it is obviously \textit{not} normal.

\section{Dixmier trace}\label{S2}

Recall that only the ideal of trace-class operators has the property that
on its {positive cone} $\mathcal{C}_{1,+}(\mathcal{H}):= \{A \in \mathcal{C}_{1}(\mathcal{H}):
A \geq 0\}$ the trace-norm is \textit{linear} since
$\|A + B\|_{1} = {\rm{Tr}} \, (A + B) = {\rm{Tr}} \, (A) + {\rm{Tr}} \, (B)=\|A\|_{1} + \|B\|_{1}$ for
$A,B \in \mathcal{C}_{1,+}(\mathcal{H})$, see Example \ref{ex6-2.1}.
Then the uniqueness of the trace-norm allows to extend the trace to the whole linear space
$\mathcal{C}_{1}(\mathcal{H})$. Imitation of this idea \textit{fails} for other symmetrically-normed ideals.

This problem motivates the Dixmier trace construction as a certain limiting procedure involving the
$\|\cdot\|_{1,\infty}$-norm.
Let $\mathcal{C}_{1,\infty,+}(\mathcal{H})$ be a {positive cone} of the Dixmier ideal. One
can try to construct on $\mathcal{C}_{1,\infty,+}(\mathcal{H})$ a
\textit{linear}, \textit{positive}, and \textit{unitarily} invariant functional (called \textit{trace}
$\mathcal{T}$)
via \textit{extension} of the limit (called Lim) of the sequence of properly normalised finite sums of the
operator $X$ singular values:
\begin{equation}\label{7-2.2}
\mathcal{T}(X) := {\rm{Lim}}_{n \rightarrow \infty} \ \frac{1}{1 +\ln(n)}\sum^n_{j=1} \ s_j(X) \ ,
\ X \in \mathcal{C}_{1,\infty,+}(\mathcal{H}) \ .
\end{equation}

First we note that since for any unitary $U: \mathcal{H} \rightarrow U$, the singular values
of $X \in \mathcal{C}_{\infty}(\mathcal{H})$ are invariant: $s_j(X) = s_j(U \, X \, U^*)$,
it is also true for the sequence
\begin{equation}\label{7-2.3}
\sigma_{n}(X) := \sum^n_{j=1} \ s_j(X) \ , \ n \in \mathbb{N} \ .
\end{equation}
Then the Lim in (\ref{7-2.2}) (if it exists) inherits the property of \textit{unitarity}.

Now we comment that positivity: $X \geq 0$, implies the positivity of eigenvalues
$\{\lambda_{j}(X)\}_{j\geq 1}$ and consequently: $\lambda_{j}(X)= s_{j}(X)$. Therefore,
$\sigma_{n}(X) \geq 0$ and the Lim in (\ref{7-2.2}) is a \textit{positive} mapping.

The next problem with the formula for $\mathcal{T}(X)$ is its \textit{linearity}. To proceed we recall that if
$P: \mathcal{H} \rightarrow P(\mathcal{H})$ is an orthogonal projection on a finite-dimensional
subspace with $\dim P(\mathcal{H}) = n$, then for any bounded operator $X \geq 0$ the (\ref{7-2.3}) gives
\begin{equation}\label{7-2.4}
\sigma_{n}(X) = \sup_{P}\ \{{\rm{Tr}} {_{\mathcal{H}}} \, (X  P): \dim P(\mathcal{H}) = n\}  \ .
\end{equation}
As a corollary of (\ref{7-2.4}) one obtains the Horn-Ky Fan inequality
\begin{equation}\label{7-2.5}
\sigma_{n}(X + Y) \leq \sigma_{n}(X) + \sigma_{n}(Y) \ , \ n \in \mathbb{N} ,
\end{equation}
valid in particular for any pair of bounded \textit{positive} compact operators $X$ and $Y$.
For $\dim P (\mathcal{H}) \leq 2 n$ one similarly gets from (\ref{7-2.4}) that
\begin{equation}\label{7-2.7}
\sigma_{2n}(X + Y) \geq \sigma_{n}(X) + \sigma_{n}(Y) \ , \ n \in \mathbb{N} \ .
\end{equation}

Motivated by (\ref{7-2.2}) we now introduce
\begin{equation}\label{7-2.8}
\mathcal{T}_{n}(X) := \frac{1}{1 +\ln(n)}\sigma_{n}(X) \ ,
\ X \in \mathcal{C}_{1,\infty,+}(\mathcal{H}) \ ,
\end{equation}
and denote  by ${\rm{Lim}}\{\mathcal{T}_{n}(X)\}_{n\in \mathbb{N}} :=
{\rm{Lim}}_{n \rightarrow \infty} \mathcal{T}_{n}(X)$
the right-hand side of the functional in (\ref{7-2.2}).
Note that by (\ref{7-2.8}) the inequalities (\ref{7-2.5}) and (\ref{7-2.7}) yield for $n \in \mathbb{N}$
\begin{eqnarray}
\mathcal{T}_{n}(X + Y) \leq \mathcal{T}_{n}(X) + \mathcal{T}_{n}(Y) \ \  , \ \
\frac{1 +\ln(2n)}{1 +\ln(n)} \ \mathcal{T}_{2n}(X + Y)\geq \mathcal{T}_{n}(X) + \mathcal{T}_{n}(Y) \ .
\label{7-2.10}
\end{eqnarray}
Since the functional Lim includes the limit $n \rightarrow \infty$, the inequalities (\ref{7-2.10})
\textit{would} give a desired linearity of the {trace} $\mathcal{T}$:
\begin{equation}\label{7-2.11}
\mathcal{T}(X + Y) = \mathcal{T}(X) + \mathcal{T}(Y)  \ ,
\end{equation}
\textit{if} one proves that the Lim$_{n \rightarrow \infty}$ in (\ref{7-2.2}) { exists} and {finite}
for $X,Y$ as well as for $X+Y$.

To this end we note that if the right-hand of (\ref{7-2.2}) exists, then one obtains (\ref{7-2.11}).
Hence the ${\rm{Lim}}\{\mathcal{T}_{n}(X)\}_{n\in \mathbb{N}}$ is a positive linear map
${\rm{Lim}}: l^{\infty}(\mathbb{N}) \rightarrow \mathbb{R}$, which defines a \textit{state}
$\omega \in \mathcal{S}(l^{\infty}(\mathbb{N}))$ on the Banach space of sequences
$\{\mathcal{T}_{n}(\cdot)\}_{n\in \mathbb{N}} \in l^{\infty}(\mathbb{N})$, such that 
$\mathcal{T}(X) = \omega (\{\mathcal{T}_{n}(X)\}_{n\in \mathbb{N}})$.
\begin{remark}\label{rem6-3.0}
Scrutinising description of $\omega(\cdot)$, we infer that its values 
${\rm{Lim}}\{\mathcal{T}_{n}(X)\}_{n\in \mathbb{N}}$ are completely determined only by the 
"tail" behaviour of the sequences $\{\mathcal{T}_{n}(X)\}_{n\in \mathbb{N}}$ as it is 
defined by ${\rm{Lim}}_{n \rightarrow \infty} \mathcal{T}_{n}(X)$. 
For example, one concludes  that the state $\omega (\{\mathcal{T}_{n}(X)\}_{n\in \mathbb{N}}) = 0$ for 
the whole set $c_0$ of sequences: $\{\mathcal{T}_{n}(X)\}_{n\in \mathbb{N}} \in c_0$, which tend to 
zero. The same is also plausible for the non-zero converging limits.     
\end{remark}

To make this description more precise we impose on the state $\omega$ the following conditions:
\begin{eqnarray*}
&&(a) \ \ \omega (\eta)\geq 0  \ , \ {\rm{for}} \  \ \forall \eta = \{\eta_n \geq 0\}_{n\in \mathbb{N}} \ , \\
&&(b) \ \ \omega (\eta) = {\rm{Lim}}\{\eta_n \}_{n\in \mathbb{N}}  =
\lim_{n \rightarrow \infty} \eta_n  \ , \ {\rm{if}} \
\{\eta_n \geq 0\}_{n\in \mathbb{N}}
\ {\rm{is \ convergent}} \ .
\end{eqnarray*}
By virtue of (a) and (b) the definitions (\ref{7-2.2}) and (\ref{7-2.8}) imply that for
$X,Y \in \mathcal{C}_{1,\infty,+}(\mathcal{H})$ one gets
\begin{eqnarray}\label{7-2.12}
&& \mathcal{T}(X) = \omega (\{ \mathcal{T}_{n}(X)\}_{n\in \mathbb{N}}) =
\lim_{n \rightarrow \infty} \mathcal{T}_{n}(X) \ ,
\\
&&\mathcal{T}(Y) = \omega (\{ \mathcal{T}_{n}(Y)\}_{n\in \mathbb{N}}) =
\lim_{n \rightarrow \infty} \mathcal{T}_{n}(Y) \ ,  \label{7-2.13} \\
&&\mathcal{T}(X + Y) = \omega (\{ \mathcal{T}_{n}(X+Y)\}_{n\in \mathbb{N}}) =
\lim_{n \rightarrow \infty} \mathcal{T}_{n}(X+Y)   \ , \label{7-2.13-1}
\end{eqnarray}
if the limits in the right-hand sides of (\ref{7-2.12})-(\ref{7-2.13-1}) exist.

Now, to ensure (\ref{7-2.11}) one has to select $\omega$ in such a way that it allows to restore the
equality in (\ref{7-2.10}), when $n \rightarrow \infty$. To this aim we impose on the state $\omega$ the
condition of \textit{dilation} $\mathfrak{D}_2$-{invariance}.

Let $\mathfrak{D}_2 : l^{\infty}(\mathbb{N}) \rightarrow l^{\infty}(\mathbb{N})$, be {dilation} mapping
$\eta \mapsto \mathfrak{D}_2(\eta)$:
\begin{equation}\label{7-2.14}
\mathfrak{D}_2: (\eta_1, \eta_2, \ldots \eta_k, \ldots) \rightarrow
(\eta_1, \eta_1, \eta_2, \eta_2, \ldots \eta_k,\eta_k, \ldots)\ , \  \forall \eta \in l^{\infty}(\mathbb{N}) \ .
\end{equation}
We say that $\omega$ is dilation $\mathfrak{D}_2$-\textit{invariant} if for any $\eta \in l^{\infty}(\mathbb{N})$
it verifies the property
\begin{equation}\label{7-2.15}
(c) \hspace{1cm}   \omega(\eta) = \omega(\mathfrak{D}_2(\eta)) \ .  \hspace{3cm}
\end{equation}

We shall discuss the question of \textit{existence} the dilation $\mathfrak{D}_2$-invariant states
(the \textit{invariant means}) on the Banach space $l^{\infty}(\mathbb{N})$ in Remark \ref{rem6-3.1}.

Let $X,Y \in \mathcal{C}_{1,\infty,+}(\mathcal{H})$. Then applying the property (c) to the sequence
$\eta = \{\xi_{2n}:=\mathcal{T}_{2n}(X+Y)\}_{n=1}^{\infty}$, we obtain
\begin{equation}\label{7-2.16}
\omega(\eta) = \omega(\mathfrak{D}_2(\eta)) =
\omega(\xi_2, \xi_2, \xi_4, \xi_4,  \xi_6, \xi_6, \ldots)\ .
\end{equation}
Note that if $\xi = \{\xi_{n} =\mathcal{T}_{n}(X+Y)\}_{n=1}^{\infty}$, then the difference of the sequences:
\begin{equation*}\label{7-2.16-1}
\mathfrak{D}_2(\eta) - \xi = (\xi_2, \xi_2, \xi_4, \xi_4,  \xi_6, \xi_6, \ldots) -
(\xi_1, \xi_2, \xi_3, \xi_4,  \xi_5, \xi_6, \ldots) \ ,
\end{equation*}
converges to \textit{zero} if $\xi_{2n} - \xi_{2n-1} \rightarrow 0$ as $n\rightarrow \infty$.
Then by virtue of (\ref{7-2.13-1}) and (\ref{7-2.16}) this would imply
\begin{equation*}\label{7-2.17}
\omega (\{ \mathcal{T}_{2n}(X+Y)\}_{n\in \mathbb{N}})=
\omega (\mathfrak{D}_2(\{ \mathcal{T}_{2n}(X+Y)\}_{n\in \mathbb{N}}))
= \omega (\{ \mathcal{T}_{n}(X+Y)\}_{n\in \mathbb{N}})\ ,
\end{equation*}
or by (\ref{7-2.13-1}):
$\lim_{n \rightarrow \infty}\mathcal{T}_{2n}(X+Y) = \lim_{n \rightarrow \infty}\mathcal{T}_{n}(X+Y)$,
which by estimates (\ref{7-2.10}) would also yield
\begin{equation}\label{7-2.17-1}
\lim_{n \rightarrow \infty}\mathcal{T}_{n}(X+Y) =
\lim_{n \rightarrow \infty}\mathcal{T}_{n}(X) + \lim_{n \rightarrow \infty}\mathcal{T}_{n}(Y) \ .
\end{equation}

Now, summarising  (\ref{7-2.12}), (\ref{7-2.13}), (\ref{7-2.13-1}) and (\ref{7-2.17-1}) we obtain the linearity
(\ref{7-2.11}) of the limiting functional $\mathcal{T}$ on the positive cone
$\mathcal{C}_{1,\infty,+}(\mathcal{H})$ if it is defined by the corresponding $\mathfrak{D}_2$-{invariant}
state $\omega$, or a dilation-invariant mean.

Therefore, to finish the proof of linearity it rests only to check that
$\lim_{n\rightarrow\infty} (\xi_{2n} - \xi_{2n-1}) = 0$. To this end we note that by definitions
(\ref{7-2.3}) and (\ref{7-2.8}) one gets
\begin{eqnarray}
\xi_{2n} - \xi_{2n-1} &=& \left[\frac{1}{\ln(2n)} - \frac{1}{\ln(2n-1)}\right]\sigma_{2n-1}(X+Y) \nonumber \\
&+& \frac{1}{\ln(2n)} s_{2n}(X+Y) \ . \label{7-2.18}
\end{eqnarray}
Since $X,Y \in \mathcal{C}_{1,\infty,+}(\mathcal{H})$, we obtain that $\lim_{n\rightarrow \infty }s_{2n}(X+Y) = 0$
and that $\sigma_{2n-1}(X+Y) = O(\ln(2n-1))$. Then taking into account that
$({1}/{\ln(2n)} - {1}/{\ln(2n-1)}) = o ({1}/{\ln(2n-1)})$ one gets that for $n\rightarrow \infty$
the right-hand side of (\ref{7-2.18}) converges to zero.

To conclude our construction of the trace $\mathcal{T}(\cdot)$ we note that by linearity (\ref{7-2.11})
one can uniquely extend this functional from the positive cone $\mathcal{C}_{1,\infty,+}(\mathcal{H})$ to the
real subspace of the Banach space $\mathcal{C}_{1,\infty}(\mathcal{H})$, and finally to the entire ideal
$\mathcal{C}_{1,\infty}(\mathcal{H})$.
\begin{definition}\label{def7-2.1}
The \textit{Dixmier trace} ${\rm{Tr}}_{\omega}(X)$ of the operator $X\in \mathcal{C}_{1,\infty,+}(\mathcal{H})$
is the value of the linear functional (\ref{7-2.2}):
\begin{equation}\label{7-2.19}
{{\rm{Tr}}}_{\omega} (X): = {\rm{Lim}}_{n \rightarrow \infty} \ \frac{\sigma_{n}(X)}{1 +\ln(n)} =
\omega (\{\mathcal{T}_{n}(X)\}_{n\in \mathbb{N}})   \ ,
\end{equation}
where ${\rm{Lim}}_{n \rightarrow \infty}$ is defined by a dilation-invariant state
$\omega \in \mathcal{S}(l^{\infty}(\mathbb{N}))$ on $l^{\infty}(\mathbb{N})$, that satisfies the properties
(a), (b), and (c). Since any self-adjoint operator $X\in \mathcal{C}_{1,\infty}(\mathcal{H})$ has the
representation: $X = X_+ - X_-$, where $X_{\pm} \in \mathcal{C}_{1,\infty,+}(\mathcal{H})$, one gets
${{\rm{Tr}}}_{\omega}  (X) = {{\rm{Tr}}}_{\omega}  (X_+) - {{\rm{Tr}}}_{\omega}  (X_-)$. Then
for arbitrary $Z\in \mathcal{C}_{1,\infty}(\mathcal{H})$ the Dixmier trace is
${\rm{Tr}}_{\omega}(Z) = {\rm{Tr}}_{\omega} ({\rm{Re}}Z) + i {\rm{Tr}}_{\omega}({\rm{Im}}Z)$.
\end{definition}

Note that if $X\in \mathcal{C}_{1,\infty,+}(\mathcal{H})$, then definition (\ref{7-2.19})
of ${\rm{Tr}}_{\omega}(\cdot)$ together with definition of the norm $\|\cdot\|_{1,\infty}$ in (\ref{6-2.26}),
readily imply the estimate ${{\rm{Tr}}}_{\omega}(X) \leq \|X \|_{1,\infty}$, which in turn yields the inequality
for arbitrary $Z$ from the Dixmier ideal $\mathcal{C}_{1,\infty}(\mathcal{H})$:
\begin{equation}\label{7-2.19-1}
|{{\rm{Tr}}}_{\omega}(Z)| \leq \|Z \|_{1,\infty} \ .
\end{equation}
\begin{remark}\label{rem6-3.1} A decisive for construction of the Dixmier trace ${\rm{Tr}}_{\omega}(\cdot)$
is the \textit{existence} of the invariant mean
$\omega \in \mathcal{S}(l^{\infty}(\mathbb{N})) \subset (l^{\infty}(\mathbb{N}))^*$.
Here the space $(l^{\infty}(\mathbb{N}))^*$ is \textit{dual} to the Banach space of bounded sequences.
Then by the Banach-Alaoglu theorem the convex set of states
$\mathcal{S}(l^{\infty}(\mathbb{N}))$ is \textit{compact} in $(l^{\infty}(\mathbb{N}))^*$ in the weak*-topology.
Now, for any $\phi \in \mathcal{S}(l^{\infty}(\mathbb{N}))$ the relation
$\phi(\mathfrak{D}_2(\cdot)) =: (\mathfrak{D}_{2}^* \phi)(\cdot)$ defines the dual
$\mathfrak{D}_{2}^*$-dilation on the set of states. By definition (\ref{7-2.14}) this map is such that
$\mathfrak{D}_{2}^*: \mathcal{S}(l^{\infty}(\mathbb{N})) \rightarrow \mathcal{S}(l^{\infty}(\mathbb{N}))$,
as well as continuous and affine (in fact linear). Then by the Markov-Kakutani theorem the dilation
$\mathfrak{D}_{2}^*$  has a fix point
$\omega \in \mathcal{S}(l^{\infty}(\mathbb{N})): \mathfrak{D}_{2}^* \omega = \omega$.
This observation justifies the existence of the \textit{invariant mean} (c) for $\mathfrak{D}_{2}$-dilation.
\end{remark}

Note that Remark \ref{rem6-3.1} has a straightforward extension to any $\mathfrak{D}_{k}$-dilation
for $k>2$, which is defined similar to (\ref{7-2.14}). Since dilations for different $k \geq 2$ \textit{commute},
the extension of the Markov-Kakutani theorem yields that the commutative family
$\mathcal{F} = \{\mathfrak{D}_{k}^* \}_{k\geq 2}$ has in $\mathcal{S}(l^{\infty}(\mathbb{N}))$ the
common fix point $\omega = \mathfrak{D}_{2}^* \omega$. Therefore, Definition \ref{def7-2.1} of the
Dixmier trace does not depend on the degree $k\geq 2$ of dilation $\mathfrak{D}_{k}$.

For more details about different constructions of \textit{invariant means} and the corresponding Dixmier trace
on $\mathcal{C}_{1,\infty}(\mathcal{H})$, see, e.g., \cite{CarSuk2006}, \cite{LoSuZa2013}.
\begin{prop}\label{pro7-2.1}
The Dixmier trace has the following properties:\\
\emph{(a)} For any bounded operator $B\in \mathcal{L}(\mathcal{H})$ and
$Z\in \mathcal{C}_{1,\infty}(\mathcal{H})$ one has
${{\rm{Tr}}}_{\omega}(Z B) = {{\rm{Tr}}}_{\omega}(B Z)$. \\
\emph{(b)} ${{\rm{Tr}}}_{\omega}(C) = 0$ for any operator $C\in \mathcal{C}_{1}(\mathcal{H})$
from the trace-class ideal, which is the closure of finite-rank operators $\mathcal{K}(\mathcal{H})$
for the $\|\cdot\|_{1}$-norm. \\
\emph{(c)} The Dixmier trace ${{\rm{Tr}}}_{\omega}:
\mathcal{C}_{1,\infty}(\mathcal{H}) \rightarrow \mathbb{C}$, is continuous in the
$\|\cdot\|_{1,\infty}$-norm.
\end{prop}
\begin{proof} (a) Since every operator $B\in \mathcal{L}(\mathcal{H})$ is a linear combination of four
unitary operators, it is sufficient to prove the equality ${{\rm{Tr}}}_{\omega}(Z U) = {{\rm{Tr}}}_{\omega}(U Z)$
for a unitary operator $U$ and moreover only for $Z\in \mathcal{C}_{1,\infty,+}(\mathcal{H})$. Then the
corresponding equality follows from the unitary invariance: $s_{j}(Z) = s_{j}(Z U)= s_{j}(U Z) = s_{j}(U Z U^*)$,
of singular values of the positive operator $Z$ for all $j \geq 1$.\\
(b) Since $C\in \mathcal{C}_{1}(\mathcal{H})$ yields $\|C\|_{1} < \infty$, definition (\ref{7-2.3}) implies
$\sigma_{n}(C) \leq \|C\|_{1}$ for any $n \geq 1$. Then by Definition \ref{def7-2.1} one gets
${{\rm{Tr}}}_{\omega}(C) = 0$. Proof of the last part of the statement is standard.\\
(c) Since the ideal $\mathcal{C}_{1,\infty}(\mathcal{H})$ is a Banach space and
${\rm{Tr}}_{\omega}: \mathcal{C}_{1,\infty}(\mathcal{H}) \rightarrow \mathbb{C}$ a linear functional it is
sufficient to consider continuity at $X =0$. Then let the sequence
$\{X_k\}_{k \geq 1} \subset \mathcal{C}_{1,\infty}(\mathcal{H})$
converges to $X =0$ in $\|\cdot\|_{1,\infty}$-topology, i.e. by (\ref{6-2.26})
\begin{equation}\label{7-2.20}
\lim_{k \rightarrow \infty} \|X_k\|_{1,\infty} =
\lim_{k \rightarrow \infty} \ \sup_{n \in \mathbb{N}} \, \frac{1}{1 +\ln(n)}\sigma_{n}(X_k) = 0 \ .
\end{equation}
Since (\ref{7-2.19-1}) implies $|{{\rm{Tr}}}_{\omega}(X_k)| \leq \|X_k\|_{1,\infty} \ $,
the assertion follows from (\ref{7-2.20}).
\hfill $\square$
\end{proof}

Therefore, the Dixmier construction gives an example of a \textit{singular} trace in the sense of Definition
\ref{def7-1.1}.

\section{Trotter-Kato product formulae in the Dixmier ideal}\label{S3}

Let $A\ge 0$ and $B\ge 0$ be two non-negative self-adjoint operators in a separable
Hilbert space $\mathcal{H}$ and let the subspace $\mathcal{H}_0 := \overline{\dom(A^{1/2}) \cap \dom(B^{1/2})}$.
It may happen that $\dom(A) \cap \dom(B)= \{0\}$, but the form-sum of these operators: $H = A \stackrel{.}{+} B$,
is well-defined in the subspace $\mathcal{H}_0 \subseteq \mathcal{H}$.

T. Kato proved in \cite{Kat1978} that under these conditions the \textit{Trotter product formula}
\begin{equation}
s-\lim_{n\to\infty}\left(e^{-tA/n}e^{-tB/n}\right)^n = e^{-tH}P_0,
\qquad t > 0,
\label{6-1.2-1}
\end{equation}
converges in the \textit{strong} operator topology \textit{away from zero} (i.e., for $t \in \mathbb{R}^{+}$),  
and \textit{locally uniformly} in $t \in \mathbb{R}^{+}$ (i.e. uniformly in $t \in[\varepsilon,T]$, 
for $0 < \varepsilon < T < +\infty \ $), to a 
\textit{degenerate} semigroup $\{e^{-tH}P_0\}_{t > 0}$. Here $P_0$ denotes the orthogonal projection 
from $\mathcal{H}$ onto $\mathcal{H}_0$. 

Moreover, in \cite{Kat1978} it was also shown that the product 
formula is true not only for the \textit{exponential} function $e^{-x}$, $x \ge 0$, but for a whole 
class of Borel measurable functions $f(\cdot)$ and $g(\cdot)$, which are defined on 
$\mathbb{R}^{+}_{0}:=[0,\infty)$ and satisfy the conditions:
\begin{eqnarray}
& & 0 \le f(x) \le 1, \qquad f(0) = 1, \qquad f'(+0) = -1, \label{6-1.3}\\
& & 0 \le g(x) \le 1, \qquad g(0) = 1, \qquad g'(+0) = -1. \label{6-1.4}
\end{eqnarray}
Namely, the main result of \cite{Kat1978} says that besides (\ref{6-1.2-1}) one also gets convergence
\begin{equation}\label{6-1.6}
\tau -\lim_{n\to\infty}\left(f(tA/n)g(tB/n)\right)^n = e^{-tH}P_0, \qquad t > 0,
\end{equation}
locally uniformly away from zero, if topology $\tau = s$.

Product formulae of the type (\ref{6-1.6}) are called the {\em Trotter-Kato product formulae} for functions
(\ref{6-1.3}), (\ref{6-1.4}), which are called the \textit{Kato functions} $\mathcal{K}$. Note that $\mathcal{K}$
is closed with respect to the \textit{products} of Kato functions.

For some particular classes of the {Kato functions} we refer to \cite{NeiZag1998}, \cite{Zag2003}.
In the following it is useful to consider instead of $f(x)g(x)$ two Kato functions: $g(x/2) f(x) g(x/2)$ and
$f(x/2) g(x) f(x/2)$, that produce the self-adjoint operator families
\begin{equation}\label{6-1.16-1}
F(t) := g(tB/2)f(tA)g(tB/2) \   {\rm{and}} \
T(t) := f(tA/2)g(tB)f(tA/2),  \ \  t \ge 0.
\end{equation}
%
%
%

Since \cite{NeiZag1990} it is known, that the \textit{lifting} of the topology of convergence in (\ref{6-1.6})
to the \textit{operator norm} $\tau = \|\cdot\|$ needs more conditions on operators
$A$ and $B$ as well as on the key Kato functions $f,g \in \mathcal{K}$. One finds a discussion and
more references on this subject in \cite{Zag2003}. Here we quote a result that will be
used below for the Trotter-Kato product formulae in the Dixmier ideal $\mathcal{C}_{1,\infty}(\mathcal{H})$.

Consider the class $\mathcal{K}_{\beta}$ of Kato-functions, which is defined in \cite{IchTam2001},
\cite{IchTamTamZag2001} as: \\
(i) Measurable functions $0 \leq h \leq 1 $ on $\mathbb{R}^{+}_{0}$, such that $h(0) = 1$, and $h'(+0) = -1$.\\
(ii) For $\varepsilon > 0$ there exists $\delta = \delta(\varepsilon) < 1$, such that
$h(s) \leq 1 - \delta(\varepsilon)$ for $s \geq \varepsilon$, and
\begin{equation*}
[h]_{\beta} := \sup_{s>0} \frac{\left|h(s) -1+s\right|}{s^{\beta}}
< \infty \ , \ \ {\rm{for}} \ \  1 < \beta \leq 2 \ .
\end{equation*}
The standard examples are: $h(s) = e^{-s}$ and $h(s)= (1 + a^{-1}s)^{-a}\ , \ a>0$.

Below we consider the class $\mathcal{K}_{\beta}$ and a particular case of generators $A$ and $B$, such
that for the Trotter-Kato product formulae the estimate of the convergence rate is \textit{optimal}.
\begin{prop}\label{IV.6-6}{\rm{\cite{IchTamTamZag2001}}}
Let $f,g \in {\mathcal{K}_{\beta}}$ with $\beta = 2$, and let $A$, $B$ be
non-negative self-adjoint operators in $\mathcal{H}$ such that the operator sum
$C := A + B $ is \textit{self-adjoint} on domain $dom(C):=dom(A)\cap dom(B)$. Then the
Trotter-Kato product formulae converge for $n\to \infty$ in the operator norm:
\begin{eqnarray*}
&& \|[f(tA/n)g(tB/n)]^n - e^{-tC}\| = O(n^{-1}) \ , \ \|[g(tB/n)f(tA/n)]^n - e^{-tC}\| = O(n^{-1})
  \ , \\
&& \|F(t/n)^n - e^{-tC}\| = O(n^{-1}) \ \ , \ \ \ \ \|T(t/n)^n - e^{-tC}\| = O(n^{-1}) \ .
\end{eqnarray*}
Note that for the corresponding to each formula {error bounds}
$O(n^{-1})$ are equal up to coefficients $\{\Gamma_j > 0\}_{j=1}^4$ and that each rate of convergence
$\Gamma_j \ \varepsilon(n)= O(n^{-1})$, $j = 1, \ldots 4$, is optimal.
\end{prop}

The first \textit{lifting } lemma yields sufficient conditions that allow to
strengthen the \textit{strong} operator convergence to the $\|\cdot\|_\phi$-norm convergence in
the the symmetrically-normed ideal $\mathcal{C}_{\phi}(\mathcal{H})$.
%
%
\begin{lem}\label{lem6-2.7}
Let self-adjoint operators: $X \in \mathcal{C}_{\phi}(\mathcal{H})$, $Y \in \mathcal{C}_{\infty}(\mathcal{H})$
and $Z \in \mathcal{L}(\mathcal{H})$. If $\{Z(t)\}_{t \ge 0}$, is a family of
self-adjoint bounded operators such that
\begin{equation}\label{6-2.49}
s-\lim_{t \to +0}Z(t) = Z \ ,
\end{equation}
then
\begin{equation}\label{6-2.50}
\lim_{r\to\infty}\sup_{t \in [0,\tau]}\|(Z(t/r) - Z)YX\|_\phi =
\lim_{r\to\infty}\sup_{t \in [0,\tau]}\|XY(Z(t/r) - Z)\|_\phi = 0 \ ,
\end{equation}
for any $\tau \in (0,\infty)$.
\end{lem}
\begin{proof} Note that (\ref{6-2.49}) yields the strong operator convergence
$s-\lim_{r\to\infty}Z(t/r) = Z$,
%
%
%
%
uniformly in $t \in [0,\tau]$. Since $Y \in \mathcal{C}_{\infty}(\mathcal{H})$, this implies
\begin{equation}\label{6-2.52}
\lim_{r\to\infty}\sup_{t \in [0,\tau]}\|(Z(t/r) - Z)Y\| = 0 \ .
\end{equation}
Since $\mathcal{C}_{\phi}(\mathcal{H})$ is a Banach space with symmetric norm (\ref{6-2.12-4}) that verifies
$\|Z X\|_\phi \leq \|Z\|\|X\|_\phi$, one gets the estimate
\begin{equation}\label{6-2.53}
\|(Z(t/r) - Z)YX\|_\phi \le \|(Z(t/r) - Z)Y\|\|X\|_\phi \ ,
\end{equation}
which together with (\ref{6-2.52}) give the prove of (\ref{6-2.50}).
\hfill $\Box$ \end{proof}


The second \textit{lifting} lemma allows to estimate the rate of convergence of the Trotter-Kato
product formula in the norm (\ref{6-2.14}) of symmetrically-normed ideal $\mathcal{C}_{\phi}(\mathcal{H})$
via the {error bound} $\varepsilon(n)$ in the operator norm due to Proposition \ref{IV.6-6}.
\begin{lem}\label{thm6-5.0}
Let $A$ and $B$ be non-negative self-adjoint operators on the separable Hilbert
space $\mathcal{H}$, that satisfy the conditions of Proposition \ref{IV.6-6}.
Let $f,g \in {\mathcal{K}_{2}}$ be such that $F(t_0) \in \mathcal{C}_{\phi}(\mathcal{H})$ for some $t_0 > 0$.

If $\Gamma_{t_0} \varepsilon(n)$, $n \in \mathbb{N}$, is the operator-norm
error bound away from $t_0 > 0$ of the Trotter-Kato product formula for
$\{f(tA)g(tB)\}_{t \ge 0}$, then for some $\Gamma_{2t_0}^{\phi} > 0$ the function
$\varepsilon_\phi(n) := \{\varepsilon([n/2]) + \varepsilon([(n+1)/2])\}$, $n \in \mathbb{N}$,
defines the error bound away from $2t_0$ of the Trotter-Kato product formula in the ideal
$\mathcal{C}_{\phi}(\mathcal{H})$:
\begin{equation}\label{6-1.16-2}
\|[f(tA/n)g(tB/n)]^n - e^{-tC}\|_{\phi} = \Gamma_{2t_0}^{\phi}\varepsilon_\phi(n) \ , \ \ \  n\to \infty.
\qquad t \ge 2t_0 \ .
\end{equation}
Here $[x] := \max\{l \in \mathbb{N}_0: l \le x\}$, for  $x \in \mathbb{R}^{+}_{0}$.
\end{lem}
\begin{proof}
To prove the assertion for the family $\{f(tA)g(tB)\}_{t \ge 0}$ we use
decompositions $n = k + m$, $k \in \mathbb{N}$ and $m = 2,3,\ldots \  $,
$n \ge 3$, for representation
\begin{eqnarray}\label{6-5.0-14}
\lefteqn{\hspace{0.0cm} (f(tA/n)g(tB/n))^n - e^{-tC} = }\\
& &  \left((f(tA/n)g(tB/n))^k - e^{-ktC/n}\right)(f(tA/n)g(tB/n))^m  \nonumber \\
& & + \ e^{-ktC/n}\left((f(tA/n)g(tB/n))^m - e^{-mtC/n}\right)\ .\nonumber
\end{eqnarray}

Since by conditions of lemma $F(t_0) \in \mathcal{C}_{\phi}(\mathcal{H})$, definition (\ref{6-1.16-1})
and representation $f(tA/n)g(tB/n))^m = f(tA/n)g(tB/n)^{1/2}F(t/n)^{m-1}g(tB)^{1/2}$
yield
\begin{equation}\label{6-5.0-16}
\|(f(tA/n)g(tB/n))^m\|_\phi \le \|F(t_0)\|_\phi \ ,
\end{equation}
for $t$ such that $t_0 \le {(m-1)}t /{n} \le (m-1)t_0$ and $m-1 \ge 1$.

Note that for self-adjoint operators $e^{-tC}$ and $F(t)$ by Araki's log-order inequality for
compact operators \cite{Ara1990} one gets for $kt/n \ge t_0$ the bound of $e^{-ktC/n}$ in the
$\|\cdot\|_\phi$-norm:
\begin{equation}\label{6-5.0-17}
\|e^{-ktC/n}\|_\phi \le \|F(t_0)\|_\phi \ .
\end{equation}
Since by Definitions \ref{def6-2.1-2} and \ref{def6-2.3} the ideal $\mathcal{C}_{\phi}(\mathcal{H})$ is
a Banach space, from (\ref{6-5.0-14})-(\ref{6-5.0-17}) we obtain the estimate
\begin{eqnarray}\label{6-5.0-18}
&&\|(f(tA/n)g(tB/n))^n - e^{-tC}\|_\phi \le \\
&&\|F(t_0)\|_\phi \ \|(f(tA/n)g(tB/n))^k - e^{-ktC/n}\|   \nonumber \\
&&+ \|F(t_0)\|_\phi \ \|(f(tA/n)g(tB/n))^m - e^{-mtC/n}\| \ ,  \nonumber
\end{eqnarray}
for $t$ such that: $(1 + {(k+1)}/{(m-1)})t_0 \le t \le nt_0$, $m \ge 2$ and $t \ge (1 + {m}/{k})t_0$.

Now, by conditions of lemma $\Gamma_{t_0} \varepsilon(\cdot)$ is the operator-norm error
bound away from $t_0$, for any interval $[a,b] \subseteq (t_0,+\infty)$. Then there exists $n_0 \in \mathbb{N}$
such that
\begin{equation}\label{6-5.0-19}
\|(f(tA/n)g(tB/n))^k - e^{-ktC/n}\| \le \Gamma_{t_0} \varepsilon(k) \ ,
\end{equation}
for ${k}t/{n} \in [a,b] \Leftrightarrow
t \in [(1 + {m}/{k})a,(1+ {m}/{k})b]$ and
\begin{equation}\label{6-5.0-20}
\|(f(tA/n)g(tB/n))^m - e^{-mtC/n}\| \le \Gamma_{t_0} \varepsilon(m) \ ,
\end{equation}
for ${m}t /{n} \in [a,b] \Leftrightarrow t \in [(1 + {k}/{m})a,(1 + {k}/{m})b]$ for all $n > n_0 $.

Setting $m := [(n+1)/2]$ and $k = [n/2]$, $n \ge 3$, we satisfy $n = k + m$ and $m \ge 2$, as well as,
$\lim_{n\to\infty} {(k+1)}/{(m-1)} = 1$, $\lim_{n\to\infty} {m}/{k} = 1$ and
$\lim_{n\to\infty} {k}/{m} = 1$. Hence, for any interval
$[\tau_0,\tau] \subseteq (2t_0,+\infty)$ we find that
$[\tau_0,\tau] \subseteq [(1 + {(k+1)}/{(m-1)})t_0, nt_0]$ for sufficiently large $n$.
Moreover, choosing $[\tau_0/2,\tau/2] \subseteq (a,b) \subseteq (t_0,+\infty)$ we satisfy
$[\tau_0,\tau] \subseteq [(1 + {m}/{k})a,(1 + {m}/{k})b]$ and
$[\tau_0,\tau] \subseteq [(1 + {k}/{m})a, (1 + {k}/{m})b]$ again for
sufficiently large $n$.

Thus, for any interval $[\tau_0,\tau] \subseteq (2t_0,+\infty)$ there
is $n_0 \in \mathbb{N}$  such that (\ref{6-5.0-18}), (\ref{6-5.0-19})
and (\ref{6-5.0-20}) hold for $t \in [\tau_0,\tau]$ and $n \ge n_0$. Therefore,
(\ref{6-5.0-18}) yields the estimate
\begin{eqnarray}\label{6-5.0-21}
&&\|(f(tA/n)g(tB/n))^n - e^{-tC}\|_\phi \le \\
&&\hspace{1cm} \Gamma_{t_0} \ \|F(t_0)\|_\phi \{\varepsilon([n/2]) + \varepsilon([(n+1)/2])\} \ , \nonumber
\end{eqnarray}
for $t \in [\tau_0,\tau] \subseteq(2t_0,+\infty)$ and $n \ge n_0$. Hence,
$\Gamma_{2t_0}^{\phi} := \Gamma_{t_0} \ \|F(t_0)\|_\phi$ and
$\Gamma_{2t_0}^{\phi} \varepsilon_\phi(\cdot)$ is an error bound in the Trotter-Kato product formula
(\ref{6-1.16-2}) away from $2t_0$ in $\mathcal{C}_{\phi}(\mathcal{H})$ for
the family $\{f(tA)g(tB)\}_{t \ge 0}$.

The lifting Lemma \ref{lem6-2.7} allows to extend the proofs for other approximants:
$\{g(tB)f(tA)\}_{t \ge 0}$, $\{F(t)\}_{t \ge 0}$ and $\{T(t)\}_{t \ge 0}$.
\hfill $\Box$
\end{proof}

Now we apply Lemma \ref{thm6-5.0} in Dixmier ideal
$\mathcal{C}_{\phi}(\mathcal{H}) = \mathcal{C}_{1, \infty}(\mathcal{H})$. This concerns the norm convergence
(\ref{6-1.16-2}), but also the estimate of the convergence rate for Dixmier traces:
\begin{equation}\label{7-3.6}
|{\rm{Tr}}{_{\omega}}(e^{-tC}) -  {\rm{Tr}}{_{\omega}}(F(t/n)^n)| \leq
\Gamma^{\omega} \varepsilon_{\omega}(n) \ .
\end{equation}
In fact, it is the same (up to $\Gamma^{\omega}$) for all Trotter-Kato approximants:
$\{T(t)\}_{t\geq 0}$, $\{f(t)g(t)\}_{t\geq 0}$, and $\{g(t)f(t)\}_{t\geq 0}$.

Indeed, since by inequality (\ref{7-2.19-1}) and  Lemma \ref{thm6-5.0} for $t \in [\tau_0,\tau]$ and $n \ge n_0$,
one has
\begin{eqnarray}
|{\rm{Tr}}{_{\omega}}(e^{-tC}) -  {\rm{Tr}}{_{\omega}}(F(t/n)^n)| \leq \|e^{-tC} - F(t/n)^n\|_{1, \infty}
\le \Gamma_{2t_0}^{\phi} \ \varepsilon_{1, \infty}(n) \ ,  \label{7-3.7}
\end{eqnarray}
we obtain for the rate in (\ref{7-3.6}): $\varepsilon_{\omega}(\cdot) = \varepsilon_{1, \infty}(\cdot)$.
Therefore, the estimate of the convergence rate for Dixmier traces (\ref{7-3.6}) and for
$\|\cdot\|_{1, \infty}$-convergence in (\ref{7-3.7}) are \textit{entirely} defined by the operator-norm
error bound $\varepsilon(\cdot)$ from Lemma \ref{thm6-5.0} and have the form:
\begin{equation}
\varepsilon_{1, \infty}(n) := \{\varepsilon([n/2]) +
\varepsilon([(n+1)/2])\} \ , \ n \in \mathbb{N} \ . \label{7-3.8}
\end{equation}

Note that for the particular case of Proposition \ref{IV.6-6}, these arguments yield for (\ref{6-5.0-21})
the explicit convergence rate asymptotics $O(n^{-1})$ for the Trotter-Kato formulae and consequently,
the same asymptotics for convergence rates of the Trotter-Kato formulae for the Dixmier trace (\ref{7-3.6}),
(\ref{7-3.7}).

Therefore, we proved in the Dixmier ideal $\mathcal{C}_{1, \infty}(\mathcal{H})$ the following assertion.

\begin{theorem}\label{pro7-1.3}
Let $f,g \in {\mathcal{K}_{\beta}}$ with $\beta = 2$, and let $A$, $B$ be
non-negative self-adjoint operators in $\mathcal{H}$ such that the operator sum
$C := A + B $ is \textit{self-adjoint} on domain $dom(C):=dom(A)\cap dom(B)$.

If $F(t_0) \in \mathcal{C}_{1, \infty}(\mathcal{H})$ for some $t_0 > 0$,
then the Trotter-Kato product formulae converge for $n\to \infty$ in the $\|\cdot\|_{1, \infty}$-norm:
\begin{eqnarray*}
&& \|[f(tA/n)g(tB/n)]^n - e^{-tC}\|_{1, \infty} = O(n^{-1}) \ ,
\ \|[g(tB/n)f(tA/n)]^n - e^{-tC}\|_{1, \infty} = O(n^{-1})
  \ , \\
&& \|F(t/n)^n - e^{-tC}\|_{1, \infty} = O(n^{-1}) \ \ , \ \ \ \ \|T(t/n)^n - e^{-tC}\|_{1, \infty} = O(n^{-1}) \ ,
\end{eqnarray*}
away from $2t_0$. The rate $O(n^{-1})$ of convergence is optimal in the sense of {\rm{\cite{IchTamTamZag2001}}}.

By virtue of (\ref{7-3.7}) the same asymptotics $O(n^{-1})$ of the convergence rate are valid
for convergence the Trotter-Kato formulae for the Dixmier trace:
\begin{eqnarray*}
&& |{\rm{Tr}}{_{\omega}}([f(tA/n)g(tB/n)]^n) - {\rm{Tr}}{_{\omega}}(e^{-tC})| = O(n^{-1}) \ , \\
&& |{\rm{Tr}}{_{\omega}}([g(tB/n)f(tA/n)]^n) - {\rm{Tr}}{_{\omega}}(e^{-tC})| = O(n^{-1}) \ , \\
&& |{\rm{Tr}}{_{\omega}}(F(t/n)^n) - {\rm{Tr}}{_{\omega}}(e^{-tC})| = O(n^{-1}) \ \ ,
\ \ \  |{\rm{Tr}}{_{\omega}}(T(t/n)^n) - {\rm{Tr}}{_{\omega}}(e^{-tC})| = O(n^{-1}) \ ,
\end{eqnarray*}
away from $2t_0$.
\end{theorem}
Optimality of the estimates in Theorem \ref{pro7-1.3} is a heritage of the optimality in Proposition \ref{IV.6-6}.
Recall that in particular this means that in contrast to the Lie product formula for \textit{bounded} generators
$A$ and $B$, the \textit{symmetrisation} of approximants $\{f(t)g(t)\}_{t\geq 0}$, and $\{g(t)f(t)\}_{t\geq 0}$
by $\{F(t)\}_{t\geq 0}$ and $\{T(t)\}_{t\geq 0}$, does not yield (in general) the improvement of the
convergence rate, see {\rm{\cite{IchTamTamZag2001}}} and discussion in \cite{Zag2005}.

We resume that the \textit{lifting} Lemmata \ref{lem6-2.7} and  \ref{thm6-5.0} are a general method to study
the convergence in symmetrically-normed ideals $\mathcal{C}_{\phi}(\mathcal{H})$ as soon as it is established in
$\mathcal{L}(\mathcal{H})$ in the operator norm topology. The crucial is to check that for any of the
\textit{key} Kato functions (e.g. for $\{F(t)\}_{t\geq 0}$) there exists $t_0 > 0$ such that
$F(t)|_{t\geq t_0} \in \mathcal{C}_{\phi}(\mathcal{H})$. Sufficient conditions for that one can find in
\cite{NeiZag1999a}-\cite{NeiZag1999c}, or in \cite{Zag2003}.

\vspace{0.5cm}
\noindent {\bf Acknowledgments.}
I am thankful to referee for useful remarks and suggestions.


\end{document}